\documentclass[12pt]{amsart}
\usepackage{a4,amssymb,amsthm,amscd,amsmath,verbatim,url,enumerate,mathdots,cleveref}
\usepackage[pdftex,usenames,dvipsnames]{color}

\newcommand{\la}{\langle}
\newcommand{\ra}{\rangle}

\newcommand{\mg}[1]{#1^{\times}}
\newcommand{\sq}[1]{#1^{\times 2}}

\renewcommand{\setminus}{\smallsetminus}

\DeclareMathOperator{\newsum}{\mathsf{\Sigma}}
\renewcommand{\sum}{\newsum}

\newcommand{\nat}{\mathbb{N}}

\newcommand{\lla}{\la\!\la}
\newcommand{\rra}{\ra\!\ra}
\newcommand{\brd}[1]{{\mathsf{Br}}_2(#1)}

\newcommand{\br}[1]{{\mathsf{Br}}(#1)}

\newcommand{\ind}{{\mathsf{ind}}}

\renewcommand{\inf}{\mathsf{inf}}
\renewcommand{\min}{\mathsf{min}}

\renewcommand{\sup}{\mathsf{sup}}

\newcommand{\tors}{\mathsf{tor}}

\newcommand{\matr}[2]{\mathbb{M}_{#1}(#2)}

\newcommand\restr[2]{{
  \left.\kern-\nulldelimiterspace 
  #1 
  \vphantom{\big|} 
  \right|_{#2} 
  }}

\renewcommand{\leq}{\leqslant}
\renewcommand{\geq}{\geqslant}
\renewcommand{\setminus}{\smallsetminus}
\renewcommand{\bmod}{\,\,\mathsf{mod}\,\,}

\newcommand{\vf}{\varphi}

\newcommand{\W}{\mathsf{W}}
\newcommand{\I}{\mathsf{I}}

\newcommand{\st}{\mathsf{st}}

\renewcommand{\log}{\mathsf{log}}

\newcommand{\N}{\mathsf{N}}
\newcommand{\T}{\mathsf{T}}

\newcommand{\trd}{\mathsf{Trd}}

\renewcommand{\deg}{\mathsf{deg}}

\newcommand{\cc}{\mathbb{C}}
\newcommand{\rr}{\mathbb{R}}
\newcommand{\zz}{\mathbb{Z}}

\newtheorem{thm*}{Theorem}

\swapnumbers
\numberwithin{equation}{section}
\newtheorem{thm}[equation]{Theorem}

\newtheorem{prop}[equation]{Proposition}

\newtheorem*{cor*}{Corollary}

\newtheorem*{conj*}{Conjecture}

\newtheorem*{qu*}{Question}

\theoremstyle{remark}

\newtheorem*{ex*}{Example}

\theoremstyle{plain}

\title{Fields with bounded Brauer $2$-torsion index} 
\dedicatory{To the memory of David W.~Lewis (1944-2021)}
\date{06.01.2023}
\author{Karim Johannes Becher}
\address{Universiteit Antwerpen, Department of Mathematics, Middelheim\-laan~1, 2020 Antwerp, Belgium.}
\email{karimjohannes.becher@uantwerpen.be}

\begin{document}


\begin{abstract}
It is shown that, over a field of characteristic not $2$, the dimension of an anisotropic quadratic Pfister form of trivial total signature is at most twice the dimension of some central division algebra of exponent $2$. The proof is based on computations with quadratic trace forms of central simple algebras.

\medskip\noindent
{\sc{Keywords:}} central simple algebra of exponent $2$, trace form, quaternion algebra, symbol length, quadratic Pfister form, Witt ring, powers of the fundamental ideal

\medskip\noindent
\sc{Classification (MSC 2020):} 11E04, 11E81, 16H05
\end{abstract}

\maketitle

\section{Introduction}

Let $F$ be a field. 
We denote by $\mg{F}$ the multiplicative group of $F$.
We denote by $\br{F}$ the Brauer group of $F$ and by $\brd{F}$ its $2$-torsion part.
Any element $\alpha\in\br{F}$ is the Brauer equivalence class of a unique central $F$-division algebra, and we denote  by $\ind(\alpha)$ the degree of this algebra.
 
We denote by $\nat$ the set of natural numbers including $0$.
Note that any element of $\brd{F}$ has index $2^n$ for some $n\in\nat$.
We define
\begin{eqnarray*}
\lambda'(F) & = & \sup\{\log_2(\ind (\alpha_{F}))\mid \alpha\in\brd{F}\}\,\in\,\nat\cup\{\infty\}\,
\end{eqnarray*}
and call this field invariant the \emph{Brauer $2$-torsion index of $F$}.

By Merkurjev's Theorem, every element of $\brd{F}$ is given by a tensor product of $F$-quaternion algebras.
For $\alpha\in\br{F}$ we define
$$\lambda^2(\alpha) \,=\, \min\,\{n\in\nat\mid \alpha\mbox{ is given by a tensor product of $n$ quaternion algebras}\}.$$
The \emph{$2$-symbol length of $F$} is defined as
\begin{eqnarray*}
\lambda^2(F) & = & \sup\{\lambda^2(\alpha)\mid \alpha\in\brd{F}\}\,\in\,\nat\cup\{\infty\}.
\end{eqnarray*}

We observe that
 $$\lambda'(F)\leq \lambda^2(F)\,.$$
A.A.~Albert showed in \cite[Theorem 6]{Alb32} that every central simple $F$-algebras of degree $4$ and exponent  $2$ is a  tensor product of two $F$-quaternion algebras. Therefore, when $\lambda'(F)\leq 2$, we have $\lambda^2(F)=\lambda'(F)$.
J.-P.~Tignol proved in \cite{Tig78} that any central simple $F$-algebra of exponent $2$ and index $8$ is equivalent to a tensor product of $4$ quaternion algebras.
Therefore, when $\lambda'(F)=3$, we have $3\leq \lambda^2(F)\leq 4$. 
No bounds on $\lambda^2(F)$ in terms of $\lambda'(F)$ are known when $\lambda'(F)>3$.

We assume in the sequel that $F$ has characteristic different from $2$.
We use standard terminology and notation in quadratic form theory, with a few exceptions which we mention explicitly.
For two quadratic forms $\vf$ and $\psi$, we write $\vf=\psi$ to indicate that they are isometric and $\vf\equiv \psi$ when they are Witt equivalent.
For $n\in\nat$ and $a_1,\dots,a_n\in\mg{F}$, we denote by $\lla a_1,\dots,a_n\rra$ the $n$-fold tensor product of binary quadratic forms 
$\la 1,-a_1\ra\otimes\dots\otimes\la 1,-a_n\ra$, which is called an \emph{$n$-fold Pfister form}.
Given a non-degenerate quadratic form $\vf$ over $F$ and $m\in\nat$, we denote the $m$-fold orthogonal sum $\vf\perp\ldots\perp\vf$ by $m\times \vf$.
We denote by $\W F$ the Witt ring of $F$ and by $\I F$ its fundamental ideal.
For $n\in\nat$, the $n$th power of the ideal $\I F$ in $\W F$ is denoted by $\I^n F$.
Note that $\I^0F=\W F$.
We denote by $\I_{\tors}F$ the torsion ideal of $\W F$.
The field $F$ is called \emph{nonreal} if $-1$ is a sum of squares in $F$, and \emph{real} otherwise.
If $F$ is nonreal, then $\W F=\I_{\tors}F$, by \cite[Corollary 6.8 and Proposition 31.4~$(6)$]{EKM08}.

Using divided power operations on the graded cohomology ring of $F(\sqrt{-1})$ with $\zz/2\zz$-coefficients, B.~Kahn \cite[Proposition~3.3]{Kah05} showed that, if $F$ is nonreal with $\lambda^2(F)<\infty$, then $\I^{2\lambda^2(F)+3}F=0$.
We will see that, under the same hypothesis, we have already $\I^{2\lambda^2(F)+2}F=0$, which Kahn could obtain only under the stronger hypothesis that $-1\in\sq{F}$. 
Since $\lambda'(F)\leq\lambda^2(F)$, 
the following main result of this article refines Kahn's statement even in the case where $-1\in\sq{F}$.
We give a formulation which also covers the case where $F$ is real.

\begin{thm*}[\Cref{T1}]\label{T:torsion-free}
Assume that $\lambda'(F)<\infty$.
Then $\I^{2\lambda'(F)+2}F(\sqrt{-1})=0$ and
$\I^{2\lambda'(F)+2}F$ is torsion-free. 
In particular, if $F$ is nonreal, then $\I^{2\lambda'(F)+2}F=0$.
\end{thm*}

Note that the statement formulated in the abstract can be derived from \Cref{T:torsion-free} via Pfister's Local-Global Principle \cite[Theorem 31.22]{EKM08}.

The result of \Cref{T:torsion-free} is optimal for the case where $F$ is nonreal, that is, we may have that $\I^{2\lambda'(F)+1}F\neq 0$.
This can be seen by the classical example, for arbitrary $n\in\nat$, of the iterated power series field in $2n+1$ variables
$$F=\cc(\!(t_1)\!)\ldots(\!(t_{2n+1})\!)\,,$$ 
for which one has $\lambda'(F)=\lambda^2(F)=n$ (see \cite[Section 4]{BH04}) and $\I^{2n+1}F\neq 0= \I^{2n+2}F$.

If $F$ is a real field, then clearly $\I^m F\neq 0$ for all $m\in\nat$.
Our second main result gives extra information in the case where $F$ is real.

For $n,r\in\nat$ we denote by $2^r\times \I^n F$ the ideal of $\W F$ consisting of $2^r$-multiples of elements of $\I^n F$,
which is contained in $\I^{r+n}F$.

\begin{thm*}[\Cref{T2}]\label{T:real}
If $0<\lambda'(F)<\infty$, then
 $\I^{2\lambda'(F)+2}F=8\times \I^{2\lambda'(F)-1}F$.
\end{thm*}
The relation formulated in \Cref{T:real} is optimal.
Before illustrating this, we point out a relation to the \emph{reduced stability index of $F$}, which is denoted by $\st(F)$ and defined as follows:
$$\st(F)\,=\,\inf \{n\in\nat\mid \I^{n+1} F\subseteq 2\times \I^n F+ \I_\tors F\}\,\in\,\nat\cup\{\infty\}\,.$$
(See \cite[Prop.~35.23]{EKM08} for this characterisation.)
Note that $\st(F)=0$ if and only if $F$ is nonreal or uniquely ordered.

\begin{cor*}
If $0<\lambda'(F)<\infty$, then $\st(F)\leq 2\lambda'(F)-1$.
\end{cor*}
\begin{proof}
Assume that $0<\lambda'(F)<\infty$ and set $n=2\lambda'(F)-1$. Then $\I^{n+3}F=8\times \I^n F$, by \Cref{T:real}.
Hence,  for any $\alpha\in\I^{n+1}F$, we have $4\times \alpha = 8\times \alpha'$ for some $\alpha'\in \I^n F$, whereby $\alpha\equiv 2\times\alpha'\bmod\I_\tors F$.
\end{proof}

Since $\lambda'(F)\leq \lambda^2(F)$, this result refines the bound $\st(F)\leq 2\lambda^2(F)-1$, which was obtained in \cite[Corollary 2.3]{BG12}.

To see that both \Cref{T:real} and the Corollary are best possible, we consider, for an arbitrary integer $n\geq 2$, the  iterated power series field
$$F=\rr(\!(t_1)\!)\ldots(\!(t_{2n-1})\!)\,.$$ 
In this case we have $\lambda'(F)=\lambda^2(F)=n$ (see \cite[Section 4]{BH04}), and the class of the  $(2n+2)$-fold Pfister form $\lla -1,-1,-1, t_1,\dots,t_{2n-1}\rra$ lies in $8\times \I^{2n-1}F$, but not in $16\times \I^{2n-2}F$. As $\I^{2n+2} F$ is torsion-free, by \Cref{T:torsion-free}, this also implies that $\st(F)=2n-1$.

In the case where $\lambda'(F)= 1$, \Cref{T:torsion-free} and \Cref{T:real} correspond to \cite[Corollary 2.8 and Corollary 2.9]{EL73}.
For all cases with $\lambda'(F)>1$, the above mentioned results are new.

\section{Trace forms and tensor product}

To obtain these results we use elementary computations with trace forms of central simple algebras.
The effectivity of this method for the study of powers of $\I F$ seems to have passed unnoticed so far.

Let $F$ be a central simple $F$-algebra $A$. We denote by $\trd_A:A\to F$ the reduced trace map.
From this $F$-linear form on $A$, one obtains a nondegenerate quadratic form 
$$\T_A:A\to F, x\mapsto \trd_A(x^2)\,,$$
which is called the \emph{trace form of $A$}.
By \cite[Lemma 2.1]{Lew94},  for any two central simple $F$-algebras $A$ and $B$, we have 
$$\T_{A\otimes B}=\T_A\otimes \T_B\,.$$

For a quadratic form $\vf$ over $F$ and $r\in\nat$, we write $2^r\vf$ for the quadratic form obtained by scaling $\vf$ with $2^r$, while we denote by $2^r\times \vf$ the $2^r$-fold orthogonal sum $\vf\perp\ldots\perp\vf$.

Given an $F$-quaternion algebra $Q$ we denote its norm form by $\N_Q$.
If we have $Q=(a,b)_F$ with $a,b\in\mg{F}$, then we obtain that $\N_Q=\lla a,b\rra=\la 1,-a,-b,ab\ra$ and $\T_Q=\la 2,2a,2b,-2ab\ra$, and therefore $\la -1\ra \perp 2\T_Q = \la 1\ra\perp -\N_Q$.

Hence, for any $F$-quaternion algebra $Q$, we have  
$$\T_Q\equiv  \la 1,1\ra \perp -2 \N_Q\,.$$
In particular, we have
$\T_{\matr{2}{F}} \equiv \la 1,1\ra$. Hence, for an arbitrary central simple $F$-algebra $A$, we obtain that
$$\T_{\matr{2}{A}}=\T_{\matr{2}{F}}\otimes \T_A\equiv 2\times \T_A\,.$$

\section{Computing Pfister forms via trace forms}

\begin{prop}\label{P:so2s-slot-trace}
Let $n\in\nat$ and $a_0,b_0,\dots,a_n,b_n\in\mg{F}$ be such that 
$a_0$ is a sum of two squares in $F$ and the $F$-algebra $\bigotimes_{i=0}^n (a_i,b_i)_F$ is not a division algebra.
Then the $(2n+2)$-fold Pfister form $\lla a_0,b_0,\dots,a_n,b_n\rra$ over $F$ is hyperbolic.
\end{prop}
\begin{proof}
Set $\rho=\lla a_0,b_0,\dots,a_n,b_n\rra$.
For $0\leq i\leq n$, we abbreviate $Q_i=(a_i,b_i)_F$, $\T_i=\T_{Q_i}$, $\N_i=\N_{Q_i}$ and we note that $\N_i=\lla a_i,b_i\rra$.
Hence $\rho=\bigotimes_{i=0}^n\N_i$.

Recall that $\T_i\equiv \la 1,1\ra\perp -2\N_i$ for $0\leq i\leq n$.
The hypothesis that $a_0$ is a sum of $2$ squares in $F$ implies that $2\times \N_0$ is hyperbolic and $-\N_0=\N_0$.
Therefore we obtain that  $\T_0\equiv \la 1,1\ra\perp 2\N_0$ and $\N_0\otimes \T_i =\N_0\otimes 2\N_i$ for $1\leq i\leq n$.

Let $B=\bigotimes_{i=1}^n Q_i$ and $A=Q_0\otimes B$. Since $\T_B=\bigotimes_{i=1}^n \T_i$, we
 obtain that   $$\N_0\otimes \T_B=  \N_0 \otimes\bigotimes_{i=1}^n (2\N_i) = (2^{n}) \rho\,.$$

Since $A$ has degree $2^{n+1}$ and is not a division algebra, we have $A=B'\otimes\matr{2}{F}$ for some central simple $F$-algebra $B'$ of degree $2^{n}$.
This yields that $$\T_0\otimes \T_B = \T_A\equiv 2\times \T_{B'}\,.$$
Since  $\T_0= \la 1,1\ra\perp 2\N_0$, we obtain that 
$$2\times \T_{B'}\equiv \T_0\otimes \T_B \equiv 2\times\T_B\perp (2^{n+1})\rho\,.$$
Therefore
$$(2^{n+1})\rho \equiv 2\times (\T_{B'}\perp-\T_B)\,.$$
As both of these forms have dimension $4^{n+1}$, we obtain that 
$$(2^{n+1})\rho = 2\times (\T_{B'}\perp-\T_B)\,.$$
As each of the forms $2\times \T_{B'}$ and $2\times \T_B$ represents $1$, we find that $\rho$ is isotropic.
Since $\rho$ is a Pfister form, we conclude that $\rho$ is hyperbolic.
\end{proof}

\begin{thm}\label{T1}
Let $n\in\nat$ be such that $\ind(\alpha)\leq 2^{n}$ for every $\alpha\in\brd{F}$.
Then $\I^{2n+2}F(\sqrt{-1})=0$ and $\I^{2n+2}F$ is torsion-free. In particular, if $F$ is nonreal, then $\I^{2n+2}F=0$.
\end{thm}
\begin{proof}
Set $m=2n+2$. 
Consider arbitrary elements $a_1,\dots,a_m\in \mg F$ and an arbitrary field extension $F'/F$.
The hypothesis implies that the $F'$-algebra $\bigotimes_{i=0}^n (a_{2i+1},a_{2i+2})_{F'}$ is not a division algebra.
If $a_1$ is a sum of two squares in $F'$, then we obtain by \Cref{P:so2s-slot-trace} that $\lla a_1,\dots,a_m\rra$ is hyperbolic over $F'$.
Hence, every $m$-fold Pfister form over $F$ becomes hyperbolic over every field extension where one of its slots becomes a sum of two squares. 
For $F'=F$, this shows that $F$ satisfies Condition $(A_m)$ in the terms of \cite[Section 35]{EKM08}.
Furthermore, since in $F'=F(\sqrt{-1})$ every element is a sum of two squares, we obtain that every $m$-fold Pfister form defined over $F$ becomes hyperbolic over $F(\sqrt{-1})$, whereby $\I^m F=2\times \I^{m-1}F$.
From these two facts together it follows by \cite[Corollary 35.27]{EKM08} that $\I^{m}F(\sqrt{-1})=0$ and $\I^mF$ is torsion-free.
\end{proof}

\begin{prop}\label{L}
Let $n\in\nat\setminus\{0\}$ and $a_1,\dots,a_{2n+1}\in\mg{F}$ be such  that the $F$-algebra $(-1,a_1)_F\otimes \bigotimes_{i=1}^n (a_{2i},a_{2i+1})_F$ is not a division algebra.
 Then there exists a $2n$-fold Pfister form $\psi$ over $F$ such that
$$2\times \lla a_1,\dots,a_{2n+1}\rra = 4\times \psi\,.$$
\end{prop}
\begin{proof}
Set $a_0=-1$.
For $0\leq i\leq n$, set $Q_i=(a_{2i},a_{2i+1})_F$, $\T_i=\T_{Q_i}$ and $\N_i=\N_{Q_i}$.
Set further $A=\bigotimes_{i=0}^n Q_i$.
Then $\T_A=\bigotimes_{i=0}^n \T_i$.
Furthermore, $\N_0=2\times \la 1,-a_1\ra$ and $\T_0\equiv 2\times \la a_1\ra$.
By the assumption, we further have that $A=B\otimes\matr{2}{F}$ for a central simple $F$-algebra $B$ of degree $2^n$.
It follows that
$$\T_A\equiv 2\times \T_B\,.$$

Set $\pi=\lla -1,a_2,\dots,a_{2n+1}\rra$, whereby $$2\times \lla a_1,\dots,a_{2n+1}\rra=\lla a_1\rra\otimes\pi\,.$$
If $-1$ is a sum of two squares in $F$, then $2\times \T_i= 2\times \N_i$ for $1\leq i\leq n$, and thus
$$2\times \T_B\equiv\T_A=\bigotimes_{i=0}^n \T_i\equiv 2\times \la a_1\ra  \otimes  \bigotimes_{i=1}^n \N_i \equiv  a_1 \pi\,,$$
whereby $a_1 \pi=2\times \T_B$ in view of the equality of the dimensions, and as $2\times \T_B$ represents $1$, it follows that $2\times \lla a_1,\dots,a_{2n+1}\rra$ is hyperbolic.

Let now $F'$ denote the function field of the quadratic form $4\times \la 1\ra$ over $F$.
Since $-1$ is a sum of two squares in $F'$, the argument of the preceding paragraph yields that $2\times \lla a_1,\dots,a_{2n+1}\rra$ becomes hyperbolic over $F'$. 
Since $4\times \la 1\ra$ is a Pfister form and $n\geq 1$, it follows by \cite[Theorem 22.5]{EKM08} 
that $4\times \la 1\ra$ is a subform of $2\times \lla a_1,\dots,a_{2n+1}\rra$ over $F$. 
By \cite[Proposition 24.1]{EKM08}, we conclude that 
$2\times \lla a_1,\dots,a_{2n+1}\rra  = 4\times \psi$ for a $2n$-fold Pfister form $\psi$ over $F$.
\end{proof}

\begin{thm}\label{T2}
Let $n\in\nat\setminus\{0\}$ be such that $\ind(\alpha)\leq 2^{n}$ for every $\alpha\in\brd{F}$.
Then $$\I^{2n+2}F=8\times \I^{2n-1}F\,.$$
\end{thm}
\begin{proof}
By \Cref{T1}, we have $\I^{2n+2}F(\sqrt{-1})=0$. It follows by \cite[Theorem 34.22]{EKM08} that $\I^{2n+2}F$ is generated by the Witt equivalence classes of forms $2\times \rho$ with $(2n+1)$-fold Pfister forms $\rho$ over $F$.
In view of the hypothesis, \Cref{L} yields for such a form $\rho$ that $2\times \rho=4\times \psi$ for a $2n$-fold Pfister form $\psi$ over $F$. 
This shows that $2\times \I^{2n+1}F=4 \times\I^{2n}F$. 
Hence $\I^{2n+2}F$ is generated by the classes of forms $4\times \psi$ for $2n$-fold Pfister forms $\psi$ over $F$.

Consider $a_1,b_1,\dots,a_{n},b_{n}\in\mg{F}$ and $\psi=\lla a_1,b_1,\dots,a_n,b_n\rra$.
For $1\leq i\leq n$, let $Q_i=(a_i,b_i)_F$, $\T_i=\T_{Q_i}$ and $\N_i=\N_{Q_i}=\lla a_i,b_i\rra$. Then $\psi=\bigotimes_{i=1}^n \N_i$.
We will compute the trace forms of the central simple $F$-algebras 
\begin{eqnarray*}A=\bigotimes_{i=1}^{n}Q_i &\mbox{ and }& A'=A\otimes (-1,-1)_F\,.\end{eqnarray*}
Since $\deg A'=2^{n+1}$, it follows from the hypothesis on $F$ that $A'= B\otimes \matr{2}{F}$ for a central simple $F$-algebra $B$ of degree $2^n$.
Note that $\T_{(-1,-1)}\equiv \la -1,-1\ra$.
Hence 
$$2\times \T_B\equiv  \T_{A'}\equiv 2\times -\T_A\,.$$
Since $\deg A=\deg B$, this yields that $$2\times \T_B= 2\times -\T_A\,.$$
As $2\times \T_A$ and $2\times \T_B$ both represent $1$, we obtain that $4\times \T_A$ is isotropic.

Since $\T_A=\bigotimes_{i=1}^n \T_{Q_i}$ and $4\times \T_{Q_i}\equiv 4\times \N_{Q_i}\bmod 8$ in $\W F$ for $1\leq i\leq n$, we obtain that $$4\times \T_A\equiv 4\times \bigotimes_{i=1}^n \N_{Q_i}\equiv 4\times \psi\bmod 8$$ in $\W F$.
Since $4\times \T_A$ is isotropic and $4\times \psi$ is a Pfister form and since both forms have the same dimension, we obtain that $4\times \psi$ becomes hyperbolic over the function field of the quadratic Pfister form $8\times \la 1\ra$ over $F$.
Since $n\geq 1$, it follows by \cite[Theorem 22.5]{EKM08}
that $8\times \la 1\ra$ is a subform of $4\times \psi$.
By \cite[Proposition 24.1]{EKM08}, this yields that $4\times \psi= 8\times \vartheta$ for an $(2n-1)$-fold Pfister form $\vartheta$ over $F$.
As this argument applies to any $2n$-fold Pfister form $\psi$ over $F$, we conclude that $4\times \I^{2n} F= 8\times \I^{2n-1}F$.
Hence $\I^{2n+2}F=8\times\I^{2n-1}F$.
\end{proof}

\section*{Acknowledgements}
The author wishes to thank  the referee for critical and very constructive comments, which helped to improve the presentation.
Further thanks for comments and suggestions are due to Saurabh Gosavi, Bruno Kahn and Jean-Pierre Tignol. 

\section*{Conflict of interest statement}
The author declares that there is no conflict of interest.

\bibliographystyle{plain}

\end{document}